\documentclass[12pt, a4paper]{article}
\usepackage{float}
\usepackage{amsmath}
\usepackage{amsthm,amscd,amsfonts,yhmath,mathrsfs,calrsfs}
\usepackage{amssymb, upref, color}

\usepackage[margin=2.3cm]{geometry}

\numberwithin{equation}{section}

\newtheorem*{theoremA}{Theorem A}
\newtheorem*{theoremB}{Theorem B}
\newtheorem{prop}{Proposition}[section]

\newtheorem{lemm}{Lemma}[section]

\newcommand{\cy}[1]{\left\langle#1\right\rangle}
\makeatletter
\renewenvironment{proof}[1][\indent\proofname]{\par
  \pushQED{\qed}%
  \normalfont \topsep6\p@\@plus6\p@\relax
  \trivlist
  \item[\hskip\labelsep
        \bfseries
    #1\@addpunct{.}]\ignorespaces
}{\popQED\endtrivlist\@endpefalse}
\makeatother
\begin{document}

\title{Character degrees of normally monomial $p$-groups of maximal class
}
\author{
Dongfang Yang, Heng Lv\footnote{Corresponding author.
\newline \indent ~~Email addresess: dfyang1228@163.com(D. Yang), lvh529@163.com(H. Lv)
\newline \indent ~~This research supported by the National Natural Science Foundation of China (No.11971391, 12071376). The first author supported by Chinese Scholarship Council.
}
\\
\small{School of Mathematics and Statistics, Southwest University, Chongqing, China}
}

\date{}
\maketitle

    \noindent\textbf{Abstract.}  A finite group $G$ is $normally ~monomial$ if all its irreducible characters are induced from linear characters of normal subgroups of $G$.
   In this paper, we determine all possible irreducible character degree sets of normally monomial 5-groups of maximal class.  Moreover,  we give an upper bound for the largest irreducible character degree of normally monomial $p$-groups of maximal class in terms of $p$.


\section{Introduction}

Let $G$ be a finite group.
In \cite{Isaacs2}, it is proved that any power of prime $p$ which contains 1 can occur as the set of irreducible character degrees of some $p$-groups of class 2.
On the other hand, the situation for the "dual" type of groups, the $p$-groups of maximal class, is still unclear.
Recall that a $p$-group of order $p^n\ge p^3$  is of maximal class if it has class $n-1$.
A $p$-group of maximal class always has a nonabelian factor group of order $p^3$ and hence
an irreducible character of degree $p$.
Especially, all 2-groups of maximal class have irreducible character degree set cd$(G)=\{1,2\}$ and all 3-groups of maximal class have irreducible character degree set cd$(G)=\{1,3\}$ or $\{1,3,3^2\}$.


A group  $G$ is $normally$ $monomial$ if every irreducible character of $G$ is induced by a linear character of a normal subgroup of $G$.
M. Slattery proved in \cite{Slattery 2010} that if $G$ is a normally monomial $5$-group of maximal class, then the irreducible character degree set cd$(G)$ is either $\{1,5,5^2,5^4\}$,  $\{5^i~|~0\le i\le k\}$ for some $k\ge 1$, or the set obtained from one of the mentioned sets by removing the degree $5^2$.
He also obtained, under additional assumptions (see \cite[Theorem 4.2]{Slattery 2010}), some further restriction for the irreducible character degree set cd$(G)$.
In this paper, we settle the problem by giving a complete list of the possible character degree sets of normally monomial 5-groups of maximal class.

\begin{theoremA}\label{ThmA}
Let $G$ be a normally monomial $5$-group of maximal class.
Then the irreducible character degree set $\mathrm{cd}(G)$ of $G$ is either $\{1,5\}$, $\{1,5,5^2\}$, $\{1,5,5^3\}$ or $\{1,5,5^2,5^3\}$.
Moreover, each one of the mentioned character degree sets occurs as cd$(G)$ for a suitable normally monomial 5-group of maximal class.
\end{theoremA}

Although the irreducible character degree set of a $p$-group of maximal class is still unclear, A. Mann \cite{Mann 2016} showed that if a $p$-group $G$ of maximal class has an irreducible character of degree larger than $p$ $(\mathrm{respectively}$~$p^2)$, then $G$ must have an irreducible character of degree at most $p^{\frac{p+1}{2}}$ $(\mathrm{respectively}~p^{\frac{p+3}{2}})$.
If $G$ is a normally monomial $p$-group of maximal class, we will show that the largest irreducible character degree of $G$ can be bounded when $G$ has order at least $p^{4p-10}$.
We denote by $b(G)$ the largest irreducible character degree of the group $G$.

\begin{theoremB}\label{thmB}
  Let $G$ be a normally monomial $p$-group of maximal class. If $|G|\ge p^{4p-10}$, then $b(G)\le p^{\frac{p+1}{2}}$.

  \end{theoremB}

\section{Notations}
Let $G$ be a $p$-group of maximal class and of order $p^n$. 
We write $G_i=\gamma_i(G)$ for the $i$-th terms of the lower central series, for $2\le i\le n-1$, and $G_1=C_G(G_2/G_4)$. Then
\[
G=G_0>G_1>G_2>\cdots >G_n=1
\]
is a chief series for $G$ and the sections $G_i/G_{i+1}$ are all of order $p$.
 The subgroup $G_1$ is called the $major~centralizer$ of $G$ in \cite{Mann 2016}.
It is a regular $p$-group and so $[A,B^p]=[A,B]^p$ for all $A,B\unlhd G_1$.
If $|G|\ge p^{p+2}$, then ${G_i}^p=G_{i+p-1}$.
It is shown \cite[Chapter 3]{Leedham} that the structure of a $p$-group $G$ of maximal class is determined by the structure of the
major centralizer $G_1$.

Assume that $n\ge p+2$.
 Let $s,s_1$ denote elements of $G$ with $s\in G-G_1$ and $s_1\in G_1-G_2$.
 For $i\ge 2$, let $s_i=[s_{i-1},s]$.
 Observe that $s_i\in G_i-G_{i+1}$ and that $G=\cy{s,s_1}G'=\cy{s,s_1}$.
 In \cite[Theorem 4]{Mann 2016}, it is proved that if $G$ is a $p$-group of maximal class, then the major centralizer $G_1$ has at most $p$ generators.
 Since a $p$-group of maximal class with order at least $p^{p+2}$ has no section which is isomorphic to $C_p\wr C_p$,
 from \cite[Theorem 4(c)]{Mann 2016}, it follows that $G_1$ has at most $p-1$ generators and hence $G_1=\cy{s_1,\cdots,s_{p-1}}$.
 Also,
 for $2\le i\le n-3$, set $H=\cy{G_i,s}$.
Then $H$ is also a $p$-group of maximal class and order $p^{n-i+1}$ and $H_j=G_{i+j}$ for $j\ge 1$.
 Therefore, we similarly have that $G_i=\cy{s_i,s_{i+1},\cdots,s_{i+p-2}}$ for $i\le n-p-1$.
  Write $C_{ij}=[s_i,s_j]$.
 Since $s_i=[s_{i-1},s]$ implies ${s_{i-1}}^s=s_{i-1}s_i$, then
 \[(*)
 \begin{split}
 {C_{ij}}^s&=[s_i,s_j]^s=[{s_i}^s,{s_j}^s]=
 [s_is_{i+1},s_js_{j+1}]\\
 &=[s_i,s_j]^{s_{i+1}s_{j+1}}[s_{i+1},s_j]^{s_{j+1}}
 [s_i,s_{j+1}]^{s_{i+1}}[s_{i+1},s_{j+1}]\\
 &=C_{ij}^{s_{i+1}s_{j+1}}C_{i+1,j}^{s_{j+1}}C_{i,j+1}^{s_{i+1}}C_{i+1,j+1}.
 \end{split}
 \]
 These notations will be consistently applied in the rest of the paper to each $p$-group of maximal class.

 For convenience, we let
 \[
  \mathcal{M}=\{G~|~G~\mathrm{is~a~normally~monomial}~p\mathrm{-group~of~maximal~class}\}.
   \]

In the following, by "group" we will always mean "finite group". We also refer to \cite{Isaacs} for standard notation in character theory of finite groups.

\section{Proofs}

We state some basic properties of normally monomial $p$-groups of maximal class.

\begin{lemm}\label{Inheritage}
Let $G\in\mathcal{M}$. If $N\unlhd G$ of index at least $p^3$, then $G/N\in\mathcal{M}$.
\end{lemm}
\begin{proof}
  Obviously $G/N$ is of maximal class.
  Take any $\chi\in \mathrm{Irr}(G/N)$.
  Since $\chi\in\mathrm{Irr}(G)$ and $G$ is normally monomial, then there exist $M\unlhd G$ and a linear character $\lambda$ of $M$ such that $\chi=\lambda^G$.
  Note that
  \[
  N\le \mathrm{ker}\chi=\mathrm{ker}(\lambda^G)\le \mathrm{ker}\lambda\le M.
  \]
  Then $\lambda\in \mathrm{Irr}(M/N)$ and thus $G/N\in\mathcal{M}$. This completes the proof.
\end{proof}

 It is easy to see that, according to the proof of Lemma \ref{Inheritage},
if $G$ is a normally monomial group with $N\unlhd G$, then $G/N$ is also normally monomial.

\begin{lemm}\label{bonner}
Let $G\in\mathcal{M}$.
If $\theta\in\mathrm{Irr}(G_i)$ for $0\le i\le n$ and $\theta(1)>1$, then $\theta^G\in \mathrm{Irr}(G)$.
\end{lemm}
\begin{proof}
  The result follows immediately from \cite[Corollary 3.4]{Bonner}.
\end{proof}





\begin{lemm}\label{lemma2}
Let $G\in\mathcal{M}$ with $\mathrm{cl}(G_1)=2$. Assume that $\mathrm{exp}({G_1}')=p$. If ${G_t}'=1$ and ${G_{t-1}}'\neq1$ for some positive integer $t$, then ${G_{t-1}}'\le [G_t,\cy{s_i}]$ for all $1\le i\le t-1$.
\end{lemm}
\begin{proof}
If $i=t-1$, since $G_{t-1}=\cy{G_t,s_{t-1}}$, it follows from ${G_t}'=1$ that
\[
  {G_{t-1}}'=[G_t\cy{s_{t-1}},G_t\cy{s_{t-1}}]={G_t}'\cdot[G_t,\cy{s_{t-1}}]=[G_t,\cy{s_{t-1}}].
\]
Next assume that $i\le t-2$.
Since $\exp({G_1}')=p$ and cl$(G_1)=2$, then $\exp(G_1/Z(G_1))=p$.
So ${s_i}^p\in Z(G_1)\le G_t$ for $1\le i\le t-2$,
which means that $G_t$ is a maximal subgroup of $\cy{G_t,s_i}$.
Thus $\cy{G_t,s_i}/[G_t,\cy{s_i}]$ is an abelian normal subgroup of $G_1/[G_t,\cy{s_i}]$.
By \cite[Theorem 6.15]{Isaacs}, we have that
\[
b(G_1/[G_t,\cy{s_i}])\le|G_1:\cy{G_t,s_i}|=|G_1:G_t|/|\cy{G_t,s_i}:G_t|=p^{t-2}.
\]
Suppose that ${G_{t-1}}'\nleq [G_t,\cy{s_i}]$ for some $1\le i\le t-2$. Let $M\le {G_{t-1}}'\cap [G_t,\cy{s_i}]$ be a $G$-invariant subgroup of maximal order.
We now claim that
$Z(G/M)\cap [G_t,\cy{s_i}]/M=1$.
Otherwise, since $G/M$ is a $p$-group of maximal class, then $|Z(G/M)|=p$ and hence $Z(G/M)\le [G_t,\cy{s_i}]/M$.
Also, $Z(G/M)\le{G_{t-1}}'/M$. Thus $Z(G/M)\le {G_{t-1}}'/M\cap [G_t,\cy{s_i}]/M$.
Let $M_1/M:=Z(G/M)$. Then $M_1\leq {G_{t-1}}'\cap [G_t,\cy{s_i}]$,
which is contrary to the maximality of $M$, as claimed.
By Lemma \ref{Inheritage}, we may assume that $M=1$ by considering the factor group $G/M$.
So $Z(G)\cap [G_t,\cy{s_i}]=1$.

Write ${G_1}'=Z(G)\times T$. Then $[G_t,\cy{s_i}]\le T$ and hence
$b(G_1/T)\le b(G_1/[G_t,\cy{s_i}])\le p^{t-2}$.
Take a nonlinear $\theta\in\mathrm{Irr}(G_1/T)$.
Then $1<\theta(1)\le p^{t-2}$ and
 by Lemma \ref{bonner} we have that $\chi:=\theta^G\in\mathrm{Irr}(G)$.
 So
 $\chi(1)= p\theta(1)\le p^{t-1}$, and since $G$ is normally monomial, it follows that there exists a linear character $\lambda$ of $G_r$ such that $\chi=\lambda^G$, where $r\le t-1$.
Notice that ${G_{t-1}}'\neq1$.
We get $Z(G)\le {G_{t-1}}'\le {G_r}'\le (\mathrm{ker}\lambda)_G=\mathrm{ker}\chi\le \mathrm{ker}\theta$, and so ${G_1}'=Z(G)\times T\le \mathrm{ker}\theta$ and hence $\theta$ is linear, contradicting the choice of $\theta$.
So ${G_{t-1}}'\le [G_t,\cy{s_i}]$ for all $1\le i\le t-2$ and the result follows.
\end{proof}

In \cite[Theorem 9.6(c)]{Berkovich}, it is proved that, for a $p$-group $G$ of maximal class with order at least $p^{p+2}$, $G$ has no normal subgroup of order $p^p$ and exponent $p$; and if $N\unlhd G$ has order $p^{p-1}$, then $N$ is of exponent $p$.

\begin{prop}\label{bound}
Let $G\in\mathcal{M}$ with $\mathrm{cl}(G_1)=2$. Then $|{G_1}'|\le p^{p-2}$.
\end{prop}
\begin{proof}
Suppose that exp$({G_1}')=p$. Let $t$ be the integer such that ${G_t}'=1$ and ${G_{t-1}}'\neq1$. By Lemma \ref{lemma2}, we have
\[
{G_{t-1}}'\le [G_t,\cy{s_1}].
\]
We now claim that ${G_1}'= [G_1,\cy{s_1}]$.
If $t=2$, it is easy to see that ${G_1}'\le [G_2,\cy{s_1}]\le [G_1,\cy{s_1}]$ and hence ${G_1}'= [G_1,\cy{s_1}]$.
If $t>2$, and let $\overline{G}=G/{G_{t-1}}'$.
It follows from Lemma \ref{Inheritage} that $\overline{G}\in \mathcal{M}$. Let $j$ be the smallest integer such that ${G_{t-j}}'\neq {G_{t-1}}'$.
Obviously $j\ge2$; and ${\overline{G}_{t-j+1}}'=1$ and ${\overline{G}_{t-j}}'\neq1$ by the minimality of $j$.
Using again Lemma \ref{lemma2},
we get
\[
{\overline{G}_{t-j_1}}'\le [\overline{G}_{t-j_1+1},\cy{\overline{s}_1}].
\]
By induction on $|G|$, we have $(+)$ ${\overline{G}_1}'=[\overline{G}_1,\cy{\overline{s}_1}]$.
Note that ${G_{t-1}}'\le[G_t,\cy{s_1}]\le[G_1,\cy{s_1}]$, and then $(+)$ implies that ${G_1}'= [G_1,\cy{s_1}]$, as claimed.

Since exp$({G_1}')=p$ and $G_1=\cy{s_1,s_2,\cdots,s_{p-1}}$, from the claim, it follows that
\[
{G_1}'=[G_1,\cy{s_1}]=\cy{[s_r,s_1]~|~2\le r\le p-1}.
\]
Therefore $|{G_1}'|\le p^{p-2}$.

Suppose that $\mathrm{exp}({G_1}')>p$. Then $|{G_1}'/({G_1}')^p|=p^{p-1}$ and by \cite[Theorem 9.6(c)]{Berkovich} we have $\mathrm{exp}({G_1}'/({G_1}')^p)=p$.
By induction on $|G|$ since $G/({G_1}')^p\in \mathcal{M}$, we have $|{G_1}'/({G_1}')^p|\le p^{p-2}$, which is contrary to $|{G_1}'/({G_1}')^p|=p^{p-1}$ and so the proof is complete.
\end{proof}

\begin{lemm}\label{00001}
Let $G\in\mathcal{M}$ with $|G|\ge p^{p+2}$. Then the following results hold.

{\rm{(1)}} If $\mathrm{cl}(G_1)=2$, then $b(G)\le p^{\frac{p+1}{2}}$;

{\rm{(2)}} If $\mathrm{cl}(G_1)=3$, then $b(G)\le p^{p}$.
\end{lemm}
\begin{proof}
(1) If cl$(G_1)=2$, by Proposition \ref{bound},  we have $|{G_1}'|\leq p^{p-2}$ and, from \cite[Theorem 9.6(c)]{Berkovich}, it follows that  $\mathrm{exp}({G_1}')=p$ and thus $\mathrm{exp}(G_1/Z(G_1))=p$, and since $G_1/Z(G_1)\unlhd G/Z(G_1)$ and $G/Z(G_1)\in \mathcal{M}$ we further have that $|G_1:Z(G_1)|\le p^{p-1}$ by \cite[Theorem 9.6(c)]{Berkovich} again.
This yields $\theta(1)^2\le p^{p-1}$ for any $\theta\in\mathrm{Irr}(G_1)$ and so $b(G_1)\le p^{\frac{p-1}{2}}$.
Since $|G:G_1|=p$, it follows that $b(G)\le p^{\frac{p+1}{2}}$.

(2) If cl$(G_1)=3$, clearly cl$(G_1/\gamma_3(G_1))=2$; and since $G/\gamma_3(G_1)\in \mathcal{M}$, from Proposition \ref{bound}, it follows  that $|{G_1}'/\gamma_3(G_1)|\le p^{p-2}$.
So $({G_1}')^p\le \gamma_3(G_1)\le Z(G_1)$ and hence
\[
{\gamma_3(G_1)}^p=[{G_1}',G_1]^p=[({G_1}')^p,G_1]\le[Z(G_1),G_1]=1,
\]
Then exp$(\gamma_3(G_1))=p$ and so $|\gamma_3(G_1)|\le p^{p-1}$.
Therefore, $|{G_1}'|\le p^{2p-3}$ and hence  ${G_1}'\le \Omega_2(G_1)$.
It follows that exp$({G_1}')\le p^2$ and so
\[
[G_p,G_p]=[(G_1)^p,(G_1)^p]=[G_1,G_1]^{p^2}=({G_1}')^{p^2}=1.
\]
Thus $G_p$ is abelian and $b(G)\le |G:G_p|= p^p$.
\end{proof}

~\par

We also recall a result.

\begin{lemm}\label{Leedham lemma2}
  Let $G$ be a $p$-group of maximal class of order $p^n$ for $p\ge 5$. Then

  {\rm(1)} If $|G|\ge p^{6p-23}$, then $G_1$ has class at most 3.

  {\rm(2)} $|\gamma_3(G_1)|\le p^{2p-8}$.

   {\rm(3)} If $n>p+1$, then $G$ has positive degree of commutativity.
\end{lemm}
\begin{proof}
  See \cite[Theorems 3.3.5, 3.3.12 and 3.4.13]{Leedham}.
  \end{proof}


The bound in Lemma \ref{Leedham lemma2}(1) will be improved in the case that $G\in\mathcal{M}$.

\begin{lemm}\label{nilpotent class}
Let $G\in\mathcal{M}$. If $|G|\ge p^{4p-10}$, then $\mathrm{cl}(G_1)\le 2$.
\end{lemm}
\begin{proof}
Obviously cl$(G_1/\gamma_3(G_1))=2$.
By Lemma \ref{Inheritage} and Proposition \ref{bound}, we have that $|{G_1}'/\gamma_3(G_1)|\le p^{p-2}$.
From Lemma \ref{Leedham lemma2}(2), it follows that $|{G_1}'|\le p^{3p-10}$.

If $|G|\ge p^{4p-10}$, 
then $|G:{G_1}'|\ge p^p=|G:G_p|$ and so ${G_1}'\le G_p$.
Suppose that $|{G_1}'|\ge p^p$. Let $N\unlhd G$ and $N\le {G_1}'$ be such that $|{G_1}'/N|=p^{p-1}$.
Then exp$({G_1}'/N)=p$ and hence $[(G_1)^p,G_1]=[G_1,G_1]^p=({G_1}')^p\le N$.
Thus $(G_1)^p/N\le Z(G_1/N)$. Moreover,
\[
{G_1}'/N\le G_p/N=(G_1)^p/N\le Z(G_1/N).
\]
So cl$(G_1/N)\le 2$.
However, since $|{G_1}'/\gamma_3(G_1)|\le p^{p-2}<p^{p-1}=|{G_1}'/N|$, then $\gamma_3(G_1)>N$ and so cl$(G_1/N)\ge3$, a contradiction.
Therefore $|{G_1}'|\le p^{p-1}$ and so exp$({G_1}')=p$.
This yields $[G_1,G_p]=[G_1,{G_1}^p]=({G_1}')^p=1$.
Hence $G_p\le Z(G_1)$. We conclude from ${G_1}'\le G_p$ that cl$(G_1)\le 2$ and the result follows.
\end{proof}

~\par
First, we are ready to prove Theorem B.\\

$\bf Proof~of~Theorem~B$ 
If $G\in\mathcal{M}$ and $|G|\ge p^{4p-10}$,
from Lemma \ref{nilpotent class}, it follows that cl$(G_1)\le 2$, and by Lemma \ref{00001} we get
 $b(G)\le p^{\frac{p+1}{2}}$ and the proof is complete.
~~~~~~~~~~~~~~~~~~~~~~$\Box$

~\par
Next, we are ready to prove Theorem A. Before proving it, we state a lemma which will be used to describe the structure of the 5-groups of maximal class and, to make it readable, we also present its proof.\\


\begin{lemm}\label{p group book}
  Let $G$ be a $p$-group of maximal class with $|G|=p^n\ge p^{p+2}$. Suppose that $s\in G-G_1$. Then $[s,g]\in G_{i+1}-G_{i+2}$ for $1\le i\le n-2$ and $g\in G_i-G_{i+1}$.
  \end{lemm}

  \begin{proof}
    Since $g\in G_i$, it follows that $[s,g]\in G_{i+1}$.
    It is enough to show $[s,g]\notin G_{i+2}$.
    If $[s,g]\in G_{i+2}$, let $\overline{G}=G/G_{i+2}$.
    Then $[\overline{s},\overline{g}]=1$.
    Since $[s,G_{i+1}]\in G_{i+2}$, it yields that $[\overline{s},\overline{G}_{i+1}]=1$.
    Note that $G_i=\cy{s,G_{i+1}}$.
    It means that $[\overline{s},\overline{G}_i]=1$.
    Thus $[s,G_i]\le G_{i+2}$ and so $s\in C_{G}(G_i/G_{i+2})$.
    However, by \cite[Definition 3.1.3]{Leedham} and \cite[Corollary 3.2.7]{Leedham}, we have that $C_{G}(G_i/G_{i+2})= C_{G}(G_1/G_2)=G_1$, contradicting $s\notin G_1$.
    The proof is complete.
  \end{proof}

  Let $G$ be a $p$-group of maximal class.
  Assume further that $G_1$ has class 2.
  Then $(*)$ in Section 2 implies that
  \[
  \begin{split}
  {C_{ij}}^s=C_{ij}C_{i+1,j}C_{i,j+1}C_{i+1,j+1}.
  \end{split}
  \]
  Let $c=C_{i+1,j}C_{i,j+1}C_{i+1,j+1}$.
  From Lemma \ref{p group book}, it is clear that $C_{ij}\in G_i-G_{i+1}$ if $c\in G_{i+1}-G_{i+2}$.\\



$\bf Proof~of~Theorem~A$
Let $G\in \mathcal{M}$.
We will first show that the only possibilities for cd$(G)$ is either $\{1,5,5^3\}$, $\{1,5\}$, $\{1,5,5^2\}$ or $\{1,5,5^2,5^3\}$.
Recalling that $5\in\mathrm{cd}(G)$, it is enough to show $b(G)\le 5^3$.
Since clearly $b(G)\le 5^3$ if $|G|\le 5^7$, we may assume that $|G|\ge 5^8$; and we claim that cl$(G_1)\le 2$.
By Lemma \ref{nilpotent class}, we only need to show the claim follows when $|G|=5^9$ or $5^8$.

Suppose that $|G|=5^9$.
If $|{G_1}'|= 5^5$, then $|\gamma_3(G_1)|=5^2$.
Since $|G_1:{G_1}'|=5^3$, it follows that $G_1/\gamma_3(G_1)=\cy{\overline{s}_1,\overline{s}_2,\overline{s}_3}$.
By Proposition \ref{bound} again, $|{G_1}'/\gamma_3(G_1)|\le 5^2$, which is contrary to $|{G_1}'/\gamma_3(G_1)|= 5^3$.
So $|{G_1}'|\le 5^4$ and exp$({G_1}')=5$.
It forces ${G_1}'\le G_5$.
Note that $[G_5,G_1]=[{G_1}^5,G_1]=({G_1}')^5=1$.
Therefore ${G_1}'\le G_5\le Z(G_1)$.

Suppose that $|G|=5^8$.
If  $|{G_1}'/\gamma_3(G_1)|=5^3$, by Proposition \ref{bound} again, $|G_1/{G_1}'|\ge 5^4$. It follows that $|G|\ge 5^9$, a contradiction.
So we get $|{G_1}'/\gamma_3(G_1)|\le 5^2$ and hence $|{G_1}'|\le 5^4$ by Lemma \ref{Leedham lemma2}(2).
If $|{G_1}'|\le 5^3$, then ${G_1}'\le G_5\le Z(G_1)$ and so cl$(G_1)\le 2$.
If $|{G_1}'|=5^4$, then $|\gamma_3(G_1)|=5^2$.
Let $\overline{G}=G/\gamma_3(G_1)$.
Since ${G_2}'=\cy{C_{23}}^G$,  from Lemma \ref{Leedham lemma2}(3), it follows that $C_{23}=[s_2,s_3]\in G_6=\gamma_3(G_1)$ and hence ${\overline{G}_2}'=1$.
By Lemma \ref{lemma2}, we have
\[
   {\overline{G}_1}'\le [\overline{G}_2, \cy{\overline{s}_1}]=\cy{\overline{C}_{12},\overline{C}_{13}}=\cy{\overline{C}_{12}}\times \cy{\overline{C}_{13}}.
\]
Therefore ${\overline{G}_1}'=\cy{\overline{C}_{12}}\times\cy{\overline{C}_{13}}$.
So ${G_1}'=\cy{C_{12}, C_{13}, \gamma_3(G_1)}$, where $C_{12}\in G_4-G_5, C_{13}\in G_5=Z(G_1)$.
 Set $a=C_{12}$.
Clearly $\gamma_3(G_1)=\cy{[s_1,a],[s_2,a]}$.
Since $s_4\in G_4-G_5$,
then $s_4=a^kc$ for some integer $k$ with $(k,5)=1$, where $c\in Z(G_1)=G_5$.
 Thus
${s_3}^s=s_3s_4=s_3a^kc$.
Moreover,
\[
{C_{23}}^s=[s_2s_3, s_3a^kc]=[s_2,s_3a^k]^{s_3}[s_3,s_3a^k]=[s_2,a^k]^{s_3}[s_2,s_3]^{a^ks_3}[s_3,a^k]=C_{23}[s_2,a^k],
    \]
where the final equality follows since $[s_2,s_3]\in G_5=Z(G_1)$ and $[s_3,a^k]\in G_8=1$.
Since $[s_2,a^k]\in G_7-1$, from Lemma \ref{p group book}, it follows that $C_{23}\in G_6-Z(G)$.
Set $N=[\cy{s_3},G_1]=\cy{C_{13},C_{23}}$ and $\widetilde{G_1}=G_1/N$.
Then $|N|=5^2$ and $N\le G_5=Z(G_1)$.
Furthermore, it is clear that cl$(\widetilde{G_1})=3$ since $\gamma_3(G_1)\nleq N$.
So $\widetilde{G_1}$ has an abelian maximal subgroup and thus cd$(\widetilde{G_1})=\{1,5\}$.
Since $\widetilde{Z(G)}\le \widetilde{G_1}'$,
it is easy to see that $G_1$ has an irreducible character $\theta$ of degree 5 such that $Z(G)\nleq \mathrm{ker}\theta$.
However, since $G$ is normally monomial, we have $\theta^G\in \mathrm{Irr}(G)$ by Lemma \ref{bonner}.
Obviously $\theta^G(1)=5^2$ and ${G_2}'\neq 1$.
It yields $Z(G)\le {G_2}'\le \mathrm{ker}(\theta^G)\le\mathrm{ker}\theta$, a contradiction, which means that cl$(G_1)\le 2$.
Hence, by Lemma \ref{00001}, we conclude that $b(G)\le 5^3$.

We now show that all character degree patterns in the statement of Theorem A do actually occur as irreducible character degree sets of suitable groups belonging to the class $\mathcal{M}$ of the normally monomial 5-groups of maximal class.
We refer to \cite{Slattery 2015} for an example of $H\in \mathcal{M}$ such that cd$(H)=\{1,5,5^3\}$.
Let now $G$ be the group constructed in Example 1(at the end of Section 4).
Then $G\in \mathcal{M}$ and cd$(G)=\{1, 5, 5^2, 5^3\}$.
Furthermore, by Lemma \ref{Inheritage}, for $i=1$ or 2, we have that $G/{G_i}'\in \mathcal{M}$.
Notice that $|{G_1}'|=5^2$ and $|{G_2}'|=5$ (see Example 1). Then $G_i/{G_i}'$ is an abelian subgroup of maximal order in $G/{G_i}'$.
It follows immediately that cd$(G/{G_1}')=\{1,5\}$ and cd$(G/{G_2}')=\{1,5,5^2\}$. The proof is complete.~~~~~~~~~~~~~~~~~~~~~~~~~~~~~~~~~~~~~~~~~~~~~~~~~~~~~~~~~~~~~~~~~~~~~~~~~~~~~~~~~~~~~~~~~~~~~~~~~~~~~~~~~~~~~~~~~~~~~~~~$\Box$






\section{An example of normally monomial 5-groups of maximal class}

In this section, we will construct an example of normally monomial 5-groups of maximal class with cd$(G)=\{1,5,5^2,5^3\}$.

We first consider the $p$-groups $G$ of maximal class with $|{G_1}'|=p$.
Recall that, in general, if $P$ is a $p$-group with $|P'|=p$, then $P=(A_1*A_2*\cdots*A_s)Z(P)$,
the central product, where $A_1,\cdots, A_s$ are minimal nonabelian $p$-subgroups (see \cite[Lemma 4.2]{Berkovich}); so $|P:Z(P)|=p^{2s}$ and cd$(P)=\{1, p^s\}$.

\begin{lemm}\label{3.22}
Let $G$ be a $p$-group of maximal class and of order $p^n$ with $n\ge p+2$.
If $|{G_1}'|=p$, let $|G_1/Z(G_1)|=p^{2k}$. Then $C_{ij}\neq 1$ for $i+j=2k+1$ and $C_{ij}=1$ for $i+j>2k+1$, where $i,j,k$ are positive integers.
\end{lemm}
\begin{proof}
Since $|{G_1}'|=p$ and $|G_1/Z(G_1)|=p^{2k}$, then $2k\le p-1$, $s_l\le Z(G_1)$ for $l\ge 2k+1$ and
\[
  G_1=\cy{s_1,s_2,\cdots,s_{2k}}Z(G_1).
\]
Suppose that  $C_{r,2k}\neq 1$ for some $2\le r \le 2k-1$.
Notice that ${C_{r,2k}}^s=C_{r,2k}$.
Then $C_{r,2k}\in Z(G)-1$.
Moreover,
\[
  {C_{r-1,2k}}^s=C_{r-1,2k}C_{r-1,2k+1}C_{r,2k}C_{r,2k+1}=C_{r-1,2k}C_{r,2k}.
\]
By Lemma \ref{p group book}, we get $C_{r-1,2k}\in Z_2(G)-Z(G)$, contradicting $C_{r-1,2k}\in{G_1}'=Z(G)$.
Hence
\[
  C_{1,2k}\neq1~\mathrm{and}~C_{2,2k}=C_{3,2k}=\cdots=C_{2k-1,2k}=1.
\]

Suppose that $2\le t\le k$.
It is easy to see that $t\le k\le \frac{p-1}{2}\le n-3$.
Then $H_t:=\cy{G_t,s}$ is also a $p$-group of maximal class.
We claim that $G_t=\cy{s_t,s_{t+1},\cdots,s_{2k-t+1}}Z(G_t)$.
Note that $G_2=\cy{s_2,s_3,\cdots,s_{2k}}Z(G_1)$.
It follows that $|G_2/Z(G_2)|=p^{2k-2}$ and hence $Z(G_2)=\cy{s_{2k}, Z(G_1)}$.
Thus $G_2=\cy{s_2,s_3,\cdots, s_{2k-1}}Z(G_2)$.
Therefore, $G_3=\cy{s_3,s_4,\cdots,s_{2k-1}}Z(G_2)$.
Similarly we have that $|G_3/Z(G_3)|=p^{2k-4}$ and hence $Z(G_3)=\cy{s_{2k-1}, Z(G_2)}$.
Thus $G_3=\cy{s_3,s_4,\cdots,s_{2k-2}}Z(G_3)$.
If $t>3$, using the similar argument,
we get
\[
  G_t=\cy{s_t,s_{t+1},\cdots,s_{2k-t+2}}Z(G_{t-1})=\cy{s_t,s_{t+1},\cdots,s_{2k-t+1}}Z(G_t),
  \]
   as claimed.
If $C_{l,2k-t+1}\neq1$ for some $t+1\le l\le 2k-t+1$, since ${C_{l-1,2k-t+1}}^s=C_{l-1,2k-t+1}C_{l,2k-t+1}$, from Lemma \ref{p group book}, it follows that $C_{l-1,2k-t+1}\in Z_2(G)-Z(G)$, a contradiction.
Thus
\[
  C_{t,2k-t+1}\neq 1~\mathrm{and}~C_{t+1,2k-t+1}=C_{t+2,2k-t+1}=\cdots=C_{2k-t,2k-t+1}=1.
  \]
Therefore
$C_{ij}\neq 1$ for $i+j=2k+1$ and $C_{ij}=1$ for $i+j>2k+1$.
\end{proof}


\begin{prop}\label{p3.1}
Let $G$ be a $p$-group of maximal class with $|{G_1}'|=p$. Then the following results hold.

(1) If $Z(G_1)$ is cyclic and $|G|>p^3$ $(p\neq 2)$,  then $G\notin \mathcal{M}$;

(2) If $Z(G_1)$ is not cyclic, then $G\in \mathcal{M}$.

\end{prop}

\begin{proof}

  Since $|{G_1}'|=p$, let $|G_1:Z(G_1)|=p^{2k}$. Then cd$(G_1)=\{1, p^k\}$.

 (1) Since $G$ has no normally cyclic subgroup of order $p^2$, then $|Z(G_1)|=p$ and so $Z(G_1)=Z(G)$.
 Then
 \[
   |G:Z(G)|=|G:Z(G_1)|=p^{2k+1}.
   \]
 It forces that $b(G)= p^k$.
 Suppose that $G$ is normally monomial and let $\theta\in\mathrm{Irr}(G_1)$ with $\theta(1)=p^k$.
 By Lemma \ref{bonner}, $\theta^G(1)=p^{k+1}\in \mathrm{cd}(G)$.
 This is contrary to $b(G)= p^k$.
 So $G\notin \mathcal{M}$.

(2)
Let $\theta$ be a nonlinear irreducible character of $G_1$.
We know that $Z(G)={G_1}'\nleq \mathrm{ker}\theta$.
This yields either ker$\theta=1$ or ker$\theta\ntrianglelefteq G$.
Since $Z(G_1)$ is not cyclic, then $G_1$ has no faithful irreducible character.
Therefore ker$\theta\ntrianglelefteq G$ for all nonlinear $\theta\in\mathrm{Irr}(G_1)$.
Since $|G:G_1|=p$, 
this forces $\theta^G\in\mathrm{Irr}(G)$.
Therefore cd$(G)=\{1,p,p^{k+1}\}$.

By Lemma \ref{3.22}, $G_{k+1}$ is abelian of order $p^{n-k-1}$.
Let
\[
  G_{k+1}=\cy{x_1}\times \cy{x_2}\times \cdots \times \cy{x_m}
\]
for some integer $m\le p-1$.
We may assume that ${G_1}'=\cy{a}\le \cy{x_1}$.
Let $\lambda \in \mathrm{Irr}(\cy{x_1})$ be such that $\lambda(a)\neq 1$.
Since $a^p=1$, then $\lambda^p(a)=1$.
By Lemma \ref{3.22} again, we have ${G_1}'=\cy{C_{ij}}$, where $i+j=2k+1$.
Next, we claim that for each $g\in G_1-G_{k+1}$ there is some element $b$ of $G_{k+1}$ such that $[g,b]\neq 1$.
We may assume $g\in G_i-G_{i+1}$ for $1\le i\le k$.
Then $g={s_i}^rg_1$ for $g_1\in G_{i+1}$ and some positive integer $r$ with $(r,p)=1$.
Let $b=s_{2k+1-i}$.
It follows that $b\in G_{k+1}$ and
\[
  [g,b]=[{s_i}^rg_1,s_{2k+1-i}]=[s_i,s_{2k+1-i}]^r[g_1,s_{2k+1-i}].
  \]
Since $g_1\in G_{i+1}$, then $g_1={s_{i+1}}^{t_1}{s_{i+2}}^{t_2}\cdots{s_{k}}^{t_{k-i}}g_2$ for $g_2\in G_{k+1}$ and some integers $t_j$ coprime with $p$.
It follows from Lemma \ref{3.22} that
\[
  [g_1,s_{2k+1-i}]={C_{i+1,2k+1-i}}^{t_1}{C_{i+2,2k+1-i}}^{t_2}\cdots {C_{k,2k+1-i}}^{t_{k-i}}[g_2,s_{2k+1-i}]=1.
  \]
  Thus $[g,b]=[{s_i}^rg_1,s_{2k+1-i}]=[s_i,s_{2k+1-i}]^r={C_{i,2k+1-i}}^r$.
  It implies that ${G_1}'=\cy{C_{i,2k+1-i}}=\cy{[g,b]}$ and so $[g,b]\neq 1$, as claimed.
 Therefore we have  $\lambda^g\neq \lambda$ since $\lambda^g(b)\neq \lambda(b)$.
Set $K=\cy{x_2}\times \cdots \times \cy{x_r}$.
It is easy to see that $(\lambda\times\mu)^g\neq \lambda\times\mu$ for each irreducible character $\mu$ of $K$.
Since $[s,b_1]\neq 1$ for $s\in G-G_1, b_1\in Z_2(G)-Z(G)$.
Similarly, $\lambda^s\neq \lambda$ and $(\lambda\times\mu)^s\neq \lambda\times\mu$.
Hence, the inertia group $I_G(\lambda\times\mu)$ contains with $G_{k+1}$.
Furthermore, $I_G(\lambda^t\times\mu)=G_{k+1}$ for $(t,p)=1$.
Let
$I=\{\theta\in\mathrm{Irr}(G_{k+1})~|~I_G(\theta)=G_{k+1}\}$.
Then
\[
  |I|=|G_{k+1}|-\frac{|G_{k+1}|}{p}=p^{n-k-1}-p^{n-k-2}.
  \]
 Note that $\theta^G\in \mathrm{Irr}(G)$ and $\theta^G(1)=p^{k+1}$ for every $\theta\in I$.
 So the set $\{\chi\in \mathrm{Irr}(G)~|~\chi=\theta^G ~\mathrm{for~ some}~ \theta\in I\}$ has size $\frac{|I|}{p^{k+1}}=\frac{p^{n-k-1}-p^{n-k-2}}{p^{k+1}}=p^{n-2k-2}-p^{n-2k-3}$. 
 On the other hand,
 since ${G_1}'\le \mathrm{ker}\chi$ for each $\chi \in \mathrm{Irr}(G)$ with $\chi(1)\le p$, then the number of irreducible characters of $G$ with degree $p^{k+1}$ is
\[
  \frac{(|G|-|G|/p)}{p^{2k+2}}=p^{n-2k-2}-p^{n-2k-3}.
  \]
 That is, all irreducible characters of $G$ of degree $p^{k+1}$ are induced from some irreducible character of $G_{k+1}$.
 Clearly cd$(G_1)=\{1,p^k\}$. 
 Therefore $G$ is a  normally monomial $p$-group and so $G\in \mathcal{M}$. The proof is complete.
\end{proof}

\begin{lemm}\label{lem1}
  Let $G$ be a 5-group of maximal class and of order at least $5^8$ with $\mathrm{cl}(G_1)=2$.
  Assume that $|{G_1}'|=5^2$.
 Then $G\in\mathcal{M}$ and $\mathrm{cd}(G)=\{1,5,5^2,5^3\}$ if and only if $G_3$ is abelian and $|{G_2}'|=5$.
 \end{lemm}
 \begin{proof}
  If $G\in\mathcal{M}$ and $b(G)=5^3$,  it follows that ${G_3}'=1$ and ${G_2}'\neq 1$.
   Let $\chi \in \mathrm{Irr}(G)$ be such that $\chi(1)=5^2$.
   Then there is a linear character $\lambda \in \mathrm{Irr}(G_2)$ such that $\chi=\lambda^G$ and thus ${G_2}'\le (\mathrm{ker}\lambda)_G=\mathrm{ker}\chi$.
   So $\chi(1)=5^2\in \mathrm{cd}(G/{G_2}')$.
   Suppose that $|{G_2}'|=5^2$.
   Then ${G_1}'={G_2}'$ and so $G/{G_2}'$ has an abelian normal maximal subgroup $G_1/{G_2}'$.
   This yields cd$(G/{G_2}')=\{1,5\}$, which is contrary to $5^2\in \mathrm{cd}(G/{G_2}')$.
   Therefore, $|{G_2}'|=5$.

  Conversely, if $G_3$ is abelian and $|{G_2}'|=5$, then ${G_2}'=[\cy{s_2}, G_3]$.
  Note that ${C_{24}}^s=C_{24}C_{25}C_{34}C_{35}=C_{24}$ and  ${C_{23}}^s=C_{23}C_{24}$.
  It follows from Lemma \ref{p group book} that $C_{24}=1$ and $Z(G)=\cy{C_{23}}={G_2}'$.
  Let $\overline{G}=G/Z(G)$.
  Then $|{\overline{G}_1}'|=5$ and hence $|\overline{G}_1/Z(\overline{G}_1)|\le 5^4$.
  Since $|\overline{G}_1|\ge 5^6$, then $|Z(\overline{G}_1)|\ge 5^2$ and so $Z(\overline{G}_1)$ is not cyclic.
  By Proposition \ref{p3.1}(2), it follows that $\overline{G}=G/Z(G)\in\mathcal{M}$ and $b(G/Z(G))=5^2$.

 Notice that ${C_{14}}^s=C_{14}C_{24}C_{15}C_{25}=C_{14}$.
 Then $C_{14}\in Z(G)$.
 Suppose that $C_{14}=1$.
 Since ${C_{13}}^s=C_{13}C_{23}C_{14}C_{24}=C_{13}C_{23}$, then $C_{13}\in Z_2(G)-Z(G)$.
 Since ${C_{12}}^s=C_{12}C_{13}C_{23}$ and $C_{13}C_{23}\in Z_2(G)-Z(G)$, then $C_{12}\in Z_3(G)-Z_2(G)$, which contradicts $|{G_1}'|=5^2$.
 Hence $C_{14}\neq1$ and so $Z(G)=\cy{C_{14}}=\cy{C_{23}}$.
 Next we claim that, for any $g\in G_1-G_3$, there exists an element $b\in G_3$ such that $Z(G)=\cy{[g,b]}$.
 If $g\in G_1-G_2$, then $g={s_1}^i{s_2}^ja_3$ where $(i,5)=(j,5)=1$ and $a_3\in G_3$. Thus
 \[
   [g,s_4]=[s_1, s_4]^i[s_2,s_4]^j[a_3,s_4]=[s_1,s_4]^i={C_{14}}^i.
   \]
  Let $b=s_4$. So $Z(G)=\cy{C_{14}}=\cy{[g,b]}$.
  If $g\in G_2-G_3$, similarly we have $g={s_2}^lb_3$ where $(l,5)=1$ and $b_3\in G_3$.
  Furthermore,
  \[
    [g,s_3]=[s_2,s_3]^l[b_3,s_3]=[s_2,s_3]^l={C_{23}}^l.
    \]
    Let $b=s_3$.
    So $Z(G)=\cy{C_{23}}=\cy{[g,b]}$ and the claim follows.
 Set $\lambda\in \mathrm{Irr}(G_3)$ such that $Z(G)\nleq$ ker$\lambda$.
 By the claim $[g,b]\notin \mathrm{ker}\lambda$.
 It follows that $\lambda^g\neq \lambda$ since $\lambda^g(b)\neq \lambda(b)$.
 This implies that $I_G(\lambda)=G_3$.
 Set $I=\{\lambda\in \mathrm{Irr}(G_3)~|~I_G(\lambda)=G_3, Z(G)\nleq \mathrm{ker}\lambda\}$ and $S=\{\lambda^G|\lambda\in I\}$.
 Then $S\subseteq \mathrm{Irr}(G)$. Moreover,
 \[
    |I|=|G_3|-\frac{|G_3|}{5}~ \mathrm{and}~|S|=\frac{|I|}{|G:G_3|}=\frac{|G_3|}{5^3}-\frac{|G_3|}{5^4}.
   \]
Therefore, $G$ has $|S|$ irreducible characters of degree $5^3$ that are induced from $G_3$.
 Note that $b(G/Z(G))=5^2$.
 Hence there are at most
 $\frac{|G|-|G/Z(G)|}{5^6}$ irreducible characters of degree $5^3$.
 Since $|G:G_3|=5^3$, then
 \[
   \frac{|G|-|G/Z(G)|}{5^6}=\frac{|G_3|}{5^3}-\frac{|G_3|}{5^4}=|S|,
   \]
   which forces $\chi(1)=5^3$ for any $\chi\in\mathrm{Irr}(G)-\mathrm{Irr}(G/Z(G))$ and each such character can be induced from $G_3$.
 Thus $G\in \mathcal{M}$ with cd$(G)=\{1,5,5^2,5^3\}$.
 \end{proof}


We now give an example of normally monomial $5$-groups $G$ of maximal class with  cd$(G)=\{1, 5, 5^2, 5^3\}$.

~\\
{\bf Example~1.} Let $N=\cy{s_1,s_2, s_3,s_4}$ be a $5$-group of class $2$, where   $N'=\cy{[s_1,s_2]}\times \cy{[s_1,s_4]}\le Z(N)$, $[s_1,s_4]=[s_2,s_3]^{-1}={s_3}^{-5}$, $[s_1,s_2]={s_2}^5$, ${s_1}^{5^2}={s_2}^{5^2}={s_3}^{5^2}={s_4}^5=1=[s_1,s_3]=[s_2,s_4]=[s_3,s_4]$.
Then $|N|=5^7$.
Let $\cy{s}$ be a cyclic subgroup, which acts on $N$ in the following way:
\[
  {s_i}^s=s_is_{i+1} ~\mathrm{for}~ i=1, 2, 3, {s_4}^s=s_4{s_1}^{-5}{s_2}^{-10}{s_3}^{-10}~\mathrm{and}~ (ab)^s=a^sb^s, \forall a,b\in N.
\]

Note that
\[
[s_1,s_2]^s=[s_1s_2, s_2s_3]=[s_1,s_2][s_2,s_3]={s_2}^{5}{s_3}^5=({s_2}^{5})^s;
\]
  \[
[s_1,s_4]^s=[s_1s_2, s_4{s_1}^{-5}{s_2}^{-10}{s_3}^{-10}]=[s_1,s_4][s_2,s_4]={s_3}^{-5}=({s_3}^{-5})^s;
\]
\[
[s_1,s_3]^s=[s_1s_2, s_3s_4]=[s_1,s_4][s_2,s_3]=1~\mathrm{and,~similarly,}~[s_2,s_4]^s=[s_3,s_4]^s=1.
\]
Hence $\cy{s}\le \mathrm{Aut}(N)$.
Next we show that $|s|=5$.
It is enough to show that ${s_1}^{s^5}=s_1$.
Since ${s_1}^s=s_1s_2$, then
 ${s_1}^{s^2}=s_1{s_2}^2s_3$ and
${s_1}^{s^3}=s_1{s_2}^3{s_3}^3s_4{s_3}^{-5}$.
Further, since $(s_is_{i+1})^n={s_i}^n{s_{i+1}}^n[s_{i+1},s_i]^{\frac{n(n-1)}{2}}$ for positive integer $n$, then
\[
  {s_1}^{s^4}=({s_1}^{s^3})^s=(s_1s_2)(s_2s_3)^3(s_3s_4)^3(s_4{s_1}^{-5}{s_2}^{-10}{s_3}^{-10}){s_3}^{-5}=s_1{s_2}^4{s_3}^6{s_4}^4{s_1}^{-5}{s_2}^{-10}{s_3}^{-5}.
  \]
  Moreover,
\[
  \begin{split}
  {s_1}^{s^5}&=({s_1}^{s^4})^s=(s_1s_2)(s_2s_3)^4(s_3s_4)^6(s_4{s_1}^{-5}{s_2}^{-10}{s_3}^{-10})^4(s_1s_2)^{-5}(s_2s_3)^{-10}{s_3}^{-5}\\
  &=s_1{s_2}^5{s_3}^{10}{s_4}^{10}{s_1}^{-20}{s_2}^{-40}{s_3}^{-40}
  {s_1}^{-5}{s_2}^{-5} {s_2}^{-10}{s_3}^{-10}{s_3}^{-5}{s_3}^{-30}\\
  &=s_1{s_2}^5{s_3}^{10}{s_2}^{-5}{s_3}^{-10}=s_1.
  \end{split}
   \]
Therefore, $|s|=5$.
 Let $G=N\rtimes \cy{s}$ be the semidirect product of $N$ and $\cy{s}$. Then $G=\cy{s,s_1}$  is a $5$-group of maximal class with order $5^8$.
 It is clear that ${G_3}'=\cy{[s_3,s_4]}^G=1$,
 ${G_2}'=\cy{[s_2,s_3]}^G=\cy{{s_3}^{-5}}$ and ${G_1}'=\cy{[s_1,s_2]}^G=\cy{{s_2}^5}\times\cy{{s_3}^{-5}}$.
 So $|{G_2}'|=5$ and $|{G_1}'|=5^2$.
 It follows from Lemma \ref{lem1} that $G\in\mathcal{M}$ and cd$(G)=\{1, 5, 5^2, 5^3\}$.~~~~~~~~~~~~~~~~~~~~~~~~~~~~~~~~~~~~~~~~$\Box$\\

{\bf Acknowledgement:}
 The authors would like to thank Professor Silvio Dolfi for his valuable comments.
 The first author is supported by China Scholarship Council (CSC), whose support is very much appreciated.

\end{document}